\documentclass[letterpaper,11pt,twoside,keywordsasfootnote,addressatend,noinfoline]{article}
\usepackage{fullpage}
\usepackage[english]{babel}
\usepackage{amssymb}
\usepackage{amsmath}
\usepackage{amsthm}
\usepackage{epsfig}
\usepackage{subfigure}
\usepackage{mathrsfs}
\usepackage{color}
\usepackage{imsart}
\usepackage{enumerate}
\usepackage{comment}

\usepackage[dvipsnames]{xcolor}

\newtheorem{theorem}{Theorem}
\newtheorem{lemma}[theorem]{Lemma}

\newtheorem{proposition}[theorem]{Proposition}

\newcommand{\N}{\mathbb{N}}
\newcommand{\Z}{\mathbb{Z}}
\newcommand{\R}{\mathbb{R}}
\newcommand{\p}{\mathbb{P}}
\newcommand{\Lat}{\mathscr{L}}

\newcommand{\norm}[1]{|\!|#1|\!|}
\newcommand{\ind}{\mathbf{1}}
\newcommand{\ep}{\epsilon}
\newcommand{\n}{\hspace*{-5pt}}

\usepackage{mathtools}

\newcommand{\E}{\mathbb{E}}

\begin{document}
\begin{frontmatter}
\title{Survival and extinction for a contact process with a density-dependent birth rate}
\runtitle{Contact process with a density-dependent birth rate}
\author{Jonas K\"oppl, Nicolas Lanchier and Max Mercer}
\runauthor{Jonas K\"oppl, Nicolas Lanchier, and Max Mercer}
\address{Weierstrass Institute \\ 10117 Berlin, Germany. \\ koeppl@wias-berlin.de}
\address{School of Mathematical and Statistical Sciences \\ Arizona State University \\ Tempe, AZ 85287, USA. \\ nicolas.lanchier@asu.edu \\ mamerce1@asu.edu}
\maketitle

\begin{abstract} \
 To study later spatial evolutionary games based on the multitype contact process, we first focus in this paper on the conditions for survival/extinction in the presence of only one strategy, in which case our model consists of a variant of the contact process with a density-dependent birth rate.
 The players are located on the~$d$-dimensional integer lattice, and have natural birth rate~$\lambda$ and natural death rate one.
 The process also depends on a payoff~$a_{11} = a$ modeling the effects of the players on each other:
 while players always die at rate one, they give birth at rate~$\lambda$ times the exponential of~$a$ times the fraction of occupied neighbors.
 In particular, the birth rate increases with the local density when~$a > 0$, in which case individuals have a beneficial effect on each other (cooperation), whereas the birth rate decreases with the local density when~$a < 0$, in which case individuals have a detrimental effect on each other (competition).
 Using standard coupling arguments to compare the process with the basic contact process, i.e., the particular case~$a = 0$, we prove that, for all payoffs~$a$, there is a phase transition from extinction to survival in the direction of the natural birth rate~$\lambda$.
 Using various block constructions among other techniques, we also prove that, for all natural birth rates~$\lambda$, there is a phase transition in the direction of~the payoff~$a$.
 This last result is in sharp contrast with the behavior of the nonspatial deterministic mean-field model in which the stability of the extinction state only depends on~$\lambda$.
 This underlines the importance of space~(local interactions) and stochasticity in our model.
\end{abstract}



\end{frontmatter}
\section{Introduction}
\label{sec:intro}
\noindent
 In the classical Harris' contact process~\cite{harris_1974}, each site of the~$d$-dimensional integer lattice is either empty or occupied by one particle.
 Particles independently give birth to one particle at rate~$\lambda$, in which case the newly created particle is sent to one of the~$2d$ nearest neighbors of the parent's site chosen uniformly at random, and die at rate one.
 When a particle is sent to a site already occupied by a particle, both particles coalesce.
 This paper is concerned with a natural variant of the contact process with a density-dependent birth rate that could be interpreted as a model of competition or cooperation depending on the choice of the parameters.
 Because we study in~\cite{koppl_lanchier_mercer} a multitype version of this model relevant in the context of evolutionary game theory, following the traditional modeling approach in this topic~\cite{maynardsmith_1982, nowak_2006}, we think of particles as players, assume that players receive a payoff~$a / 2d$ from each of their~(occupied) neighbors, and interpret the players' payoff as their fitness~(fertility).
 More precisely, like for the basic contact process, the state at time $t$ is a spatial configuration
 $$ \xi_t : \Z^d \longrightarrow \{0, 1\} \quad \hbox{where} \quad \hbox{0 = empty \ \ and \ \ 1 = occupied}. $$
 The payoff of the player at site~$x$ is given by~$a$ times the fraction~$f_1 (x, \xi_t)$ of occupied neighbors of site~$x$.
 To turn the payoff, which might be negative, into a birth rate, we assume that the player gives birth at rate~$\lambda \,h (a f_1 (x, \xi_t))$ where
\begin{equation}
\label{eq:mapping}
  h : \R \longrightarrow \R_+ \quad \hbox{is increasing and satisfies} \ h (0) = 1.
\end{equation}
 The monotonicity of the function~$h$ models the fact that, the larger the payoff, the more beneficial/less detrimental the interactions with other players.
 The condition~$h (0) = 1$ also indicates that, when~$a = 0$, the interactions are neutral, in which case the process reduces to the basic contact process.
 In particular, when~$a > 0$, players have a beneficial effect on each other~(cooperation) and the birth rate increases with the local density, whereas when~$a < 0$, players have a detrimental effect on each other~(competition) and the birth rate decreases with the local density.
 To fix the ideas, we assume from now on that~$h = \exp$, i.e., the player at site~$x$ gives birth at rate
\begin{equation}
\label{eq:birth-rate}
\Phi (x, \xi_t) = \lambda \exp (\hbox{payoff}) = \lambda \exp (a f_1 (x, \xi_t))
\end{equation}
 and dies at rate~1,
 but our results easily extend to any function~$h$ such that
 $$ \eqref{eq:mapping} \ \hbox{holds}, \qquad h (r) \to 0 \ \ \hbox{as} \ \ r \to - \infty, \qquad h (r) \to \infty \ \ \hbox{as} \ \ r \to + \infty. $$
 Snapshots of the contact process with birth rate~\eqref{eq:birth-rate} and different values of the parameters~$\lambda$ and~$a$ are shown in Figure~\ref{fig:payoff}.
\begin{figure}[t]
\centering
\scalebox{0.78}{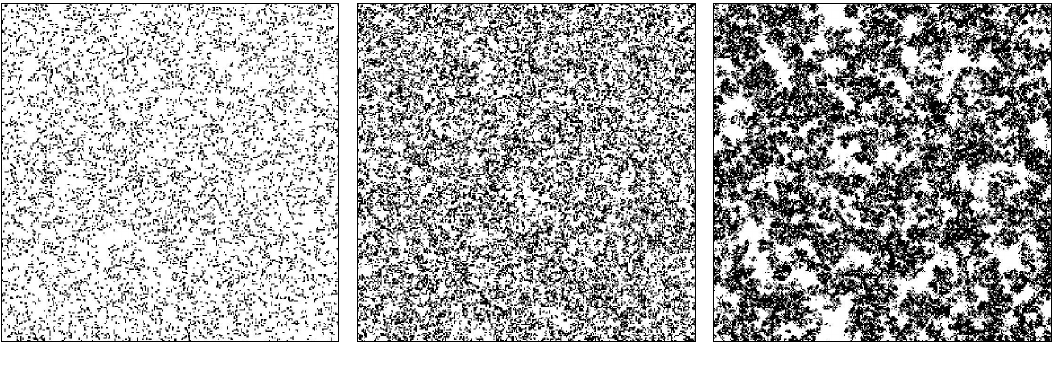}
\caption{\upshape
 Snapshots at time~1000 of the two-dimensional contact process with various birth parameters~$\lambda$ and payoffs~$a$.
 Middle: When~$a = 0$, the process reduces to the basic contact process.
 Left: When~$a < 0$, small clusters have a reduced birth rate, which results in a more scattered configuration.
 Right: In contrast, when~$a > 0$, small clusters have an increased birth rate, which results in stronger spatial correlations.}
\label{fig:payoff}
\end{figure}
 As previously mentioned, our process can be viewed, in the context of population ecology, as a model of competition when~$a < 0$ and cooperation when~$a > 0$, with the strength of competition/cooperation increasing with~$|a|$.
 The results in this paper are also used in~\cite{koppl_lanchier_mercer} where we study a multitype version of interest in evolutionary game theory.
 With the exception of the process in~\cite{lanchier_2019} that was only designed to model the prisoner's dilemma rather than general games, all the interacting particle systems of interest in evolutionary game theory that were introduced in the probability literature consist of natural variants of the voter model~\cite{clifford_sudbury_1973,holley_liggett_1975}.
 The objective of~\cite{koppl_lanchier_mercer} is to initiate the study of evolutionary games based on the more realistic Neuhauser's multitype contact process~\cite{neuhauser_1992}, looking at a hybrid of the neutral multitype contact process characterized by a natural birth rate and the popular birth-death updating process~\cite{ohtsuki_al_2006} characterized by a payoff matrix.
 We point out, however, that the process in this paper does not fall in the framework of evolutionary game theory since all the individuals are of the same type/follow the same strategy. \\
\indent
 Because the birth rate in~\eqref{eq:birth-rate} is nondecreasing with respect to the natural birth rate~$\lambda$ and the payoff coefficient~$a$, it is expected that the probability of survival, defined as the probability that, starting with a single player, there will be at least one player at all times, also is nondecreasing with respect to~$\lambda$ and~$a$.
 Standard coupling arguments, however, fail to prove this result in the presence of competition~$a < 0$.
 Indeed, increasing the birth rate increases the local density of occupied sites, which decreases the birth rate of the surrounding players, so standard monotonicity arguments cannot be applied.
 Nevertheless, we can show that if the ``larger'' of two processes is cooperative, or at least neutral, then its long-term probability of survival is higher.
 More precisely, letting~$\xi_t^1$ and~$\xi_t^2$ be the processes with parameter pairs~$(\lambda_1, a_1)$ and~$(\lambda_2, a_2)$, respectively, we will prove the existence of a coupling such that
 $$ \begin{array}{rcl} \lambda_1 \leq \lambda_2, \ \ a_1 \vee 0 \leq a_2, \ \ \xi^1 \subset \xi^2 & \Longrightarrow & \xi_t^1 \subset \xi^2 \quad \forall t. \end{array} $$
 In the particular case where the initial configurations are equal, the previous implication is called ``monotonicity'', while in the particular case where the parameters are equal, it is referred to as ``attractiveness''.
 It follows from this coupling that,
 letting~$\p_{\lambda, a}^{\xi}$ be the law of the process with parameters~$\lambda$ and~$a$, starting from~$\xi_0 = \xi$,
\begin{equation}
\label{eq:monotone}
\begin{array}{rcl}
\lambda_1 \leq \lambda_2, \ \ a_1 \vee 0 \leq a_2, \ \ \xi^1 \subset \xi^2 & \Longrightarrow &
\p_{\lambda_1, a_1}^{\xi^1} [\xi_t \neq \varnothing \ \forall t] \leq \p_{\lambda_2, a_2}^{\xi^2} [\xi_t \neq \varnothing \ \forall t], \end{array}
\end{equation}
 showing monotonicity of the long-term probability of survival with respect to the parameters and the initial configuration.
 The implication~\eqref{eq:monotone} also shows that there is at most one phase transition from extinction to survival in the direction of each of the two parameters in the presence of cooperation.
 In addition, the natural birth rate~$\lambda$ being fixed, the probability of survival in the presence of cooperation~$a > 0$ is no less than the probability of survival in the presence of competition~$a < 0$.
 To indeed prove the existence of a phase transition, observe that, in the neutral case~$a = 0$, the birth rate~\eqref{eq:birth-rate} becomes~$\lambda$ regardless of the configuration, showing that the process reduces to the basic contact process.
 Letting~$\lambda_c = \lambda_c (\Z^d)$ denote the critical value of the contact process on the~$d$-dimensional lattice, and using other couplings, we can prove that
\begin{theorem}
\label{th:coupling}
 For every fixed~$- \infty < a < \infty$, the process
 $$ \begin{array}{rrrclcrcl}
    \hbox{(a)} & \hbox{survives when} & (a > 0 & \n \hbox{and} \n & \lambda > \lambda_c)    & \hbox{or}
                                      & (a < 0 & \n \hbox{and} \n & \lambda > \lambda_c \,e^{- a (1 - 1/2d)}), \vspace*{2pt} \\
    \hbox{(b)} & \hbox{dies out when} & (a < 0 & \n \hbox{and} \n & \lambda \leq \lambda_c) & \hbox{or}
                                      & (a > 0 & \n \hbox{and} \n & \lambda \leq \lambda_c \,e^{- a (1 - 1/2d)}). \end{array} $$
\end{theorem}
\noindent
 The theorem implies that, for all fixed~$a$, there is~(at least) one phase transition from extinction to survival at a nondegenerate critical value
 $$ \lambda_c (a) = \inf \{\lambda : \p_{\lambda, a}^{\{0 \}} [\xi_t \neq \varnothing \ \forall t] > 0 \} \in (0, \infty). $$
 More precisely, we have the upper and lower bounds
 $$ \begin{array}{rcl}
      0 < \lambda_ c e^{- a (1 - 1/2d)} < \lambda_c (a) < \lambda_c < \infty & \hbox{for all} & a > 0, \vspace*{4pt} \\
      0 < \lambda_ c < \lambda_c (a) < \lambda_c e^{- a (1 - 1/2d)} < \infty & \hbox{for all} & a < 0. \end{array} $$
 In addition, taking~$a_1 = a_2 > 0$ and~$\xi^1 = \xi^2 = \{0 \}$ in~\eqref{eq:monotone} shows that the long-term probability of survival is monotone with respect to the natural birth rate~$\lambda$, from which it follows that the phase transition for~$a > 0$ is unique, i.e., for all~$a > 0$,
 $$ \begin{array}{rcl}
    \lambda > \lambda_c (a) & \Longrightarrow & \p_{\lambda, a}^{\{0 \}} [\xi_t \neq \varnothing \ \forall t] > 0, \vspace*{4pt} \\ 
    \lambda < \lambda_c (a) & \Longrightarrow & \p_{\lambda, a}^{\{0 \}} [\xi_t \neq \varnothing \ \forall t] = 0. \end{array} $$
 Proving that, for each fixed birth rate~$0 < \lambda < \infty$, there exists a phase transition in the direction of the payoff~$a$ is more complicated.
 Looking first at survival, even when~$\lambda > 0$ is small, for~$a < \infty$ large positive, adjacent players have a large birth rate due to cooperation.
 In particular, it can be proved that, with high probability, a small block of players quickly doubles in size.
 Using also attractiveness when the payoff is positive, and a block construction, implies survival when~$a$ is large positive.
\begin{theorem}[Survival]
\label{th:survival}
 For every birth rate~$\lambda > 0$, there exists an~$a_+ = a_+ (\lambda, d) < + \infty$ such that the process survives for all payoffs~$ a > a_+$.
\end{theorem}
\noindent
 In contrast, even when~$\lambda < \infty$ is large, for~$a = - \infty$, adjacent players cannot give birth due to competition, and it can be proved that the family generated by a single player decays exponentially in space and time.
 Using another block construction and a perturbation argument implies extinction when~$a$ is large negative.
\begin{theorem}[Extinction]
\label{th:extinction}
 For every birth rate~$\lambda < \infty$, there exists an~$a_- = a_- (\lambda, d) > - \infty$ such that the process dies out for all payoffs~$a < a_-$.
\end{theorem}
\noindent
 We refer to the phase diagram in Figure~\ref{fig:phase} for a summary and visualization of our main results.
\begin{figure}[t!]
\centering
\scalebox{0.40}{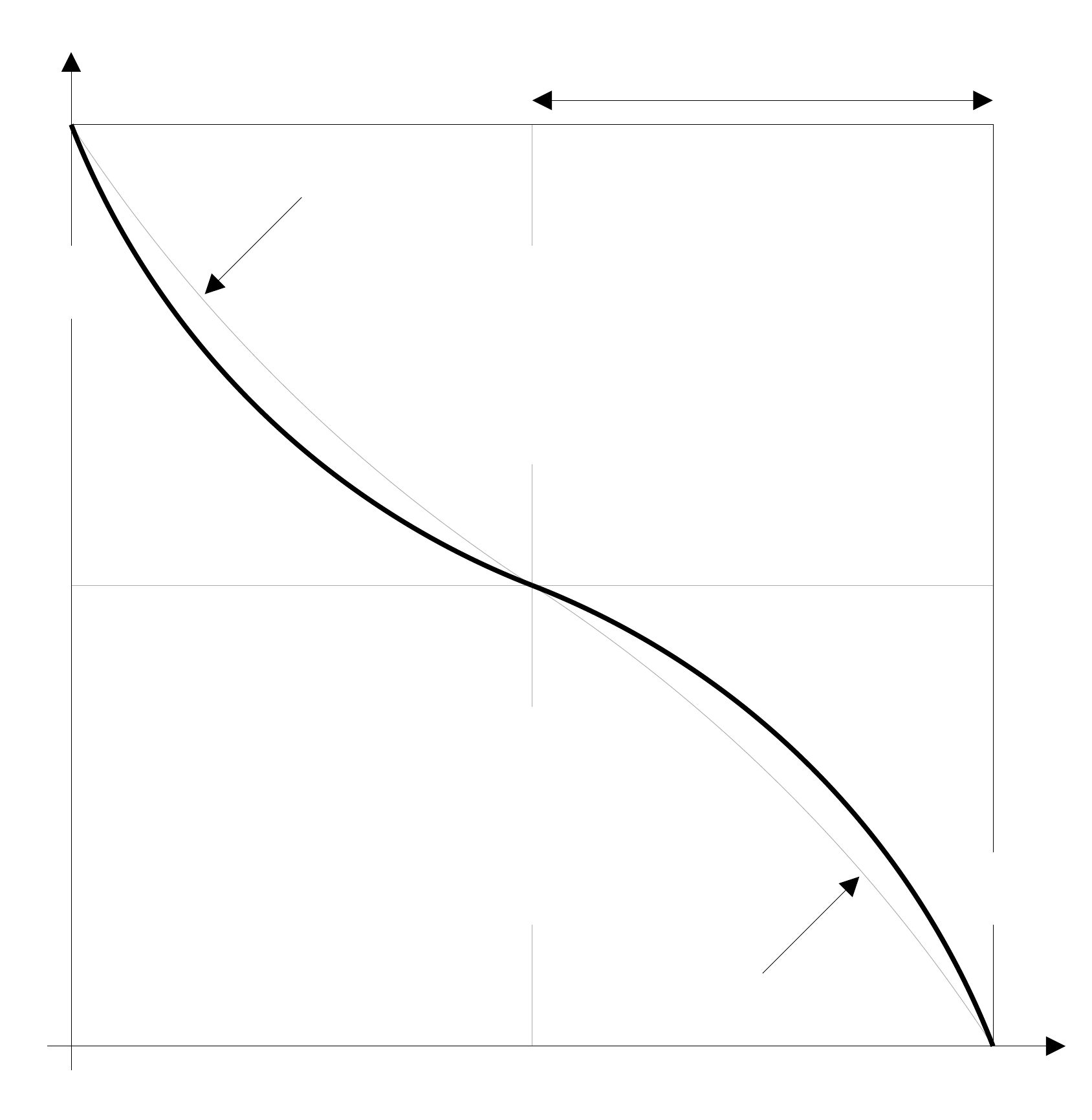}
\caption{\upshape{
 Phase structure of the contact process with a density-dependent birth rate.}}
\label{fig:phase}
\end{figure}
 Before going into the proofs, we point out that Theorem~\ref{th:extinction} cannot be deduced from general ergodicity results like the~$(M - \ep)$-criterion~\cite[Theorem I.4.1]{liggett_1985}, which asserts that a process is ergodic if the quantity~$M$, measuring how much the transition rate at a site~$x$ depends on the configuration on~$\Z^d \setminus \{x\}$, is less than the quantity~$\ep$, which can be seen as the underlying minimal transition rate of the process.
 Just like for the classical contact process, we have~$\ep = 1$, while~$M = M (\lambda, a)$ is lower bounded by the corresponding value for the contact process, i.e.,~$M \geq 2d \lambda$.
 In particular, the~$(M - \varepsilon)$-criterion can at most be useful in subcritical regime of the contact process but does not tell us anything about what happens for large~$\lambda$.
 We also note that, although the technical details in the proofs of Theorems~\ref{th:survival} and~\ref{th:extinction} differ significantly, the intuition behind both results is the same.
 Even if~$\lambda$ is very small, once a player gives birth, this player and its offspring form an adjacent pair with an arbitrarily large birth rate when~$a$ is large positive.
 Similarly, even if~$\lambda$ is very large, once a player gives birth, this player and its offspring form an adjacent pair with an arbitrarily small birth rate when~$a$ is large negative. 
 In both cases, the conclusion~(survival/extinction) is due to the presence of local interactions: the players place their offspring in their neighborhood, while their payoff is also determined by their neighbors.
 In particular, the two theorems are expected to fail in the absence of local interactions.
 Indeed, a simple analysis of the nonspatial deterministic mean-field model shows that whether the trivial extinction fixed point~0 is stable or unstable, and so whether the population dies out or survives starting at low density, depends on~$\lambda$ but not on~$a$. \\
\indent
 The rest of the paper is devoted to the proofs.
 Section~\ref{sec:mean-field} gives a brief analysis of the mean-field model, focusing on the local stability of the trivial fixed point, but also on the size of its basin of attraction.
 Section~\ref{sec:coupling} relies on various couplings to prove monotonicity and attractiveness in the case where~$a \geq 0$.
 Similar couplings are used to compare the process with the basic contact process, and deduce Theorem~\ref{th:coupling}.
 Section~\ref{sec:survival} uses a block construction to prove Theorem~\ref{th:survival}.
 Finally, Section~\ref{sec:extinction} establishes some exponential decay to deduce extinction of the process with~$a = - \infty$ from a block construction.
 Theorem~\ref{th:extinction} is then deduced by using in addition a perturbation argument.


\section{Mean-field model}
\label{sec:mean-field}
\noindent
 This section gives a brief analysis of the nonspatial deterministic mean-field model, which describes the process in the large population limit when the system is homogeneously mixing.
 To derive this model formally, assume that the population is located on the complete graph with~$N$ vertices, and let~$|\xi_t|$ be the number of individuals at time~$t$.
 Each individual dies at rate one and gives birth at rate~$\lambda e^{a |\xi_t|/N}$, and an offspring is sent to an empty site with probability~$1 - |\xi_t|/N$.
 In particular, it follows from the superposition property and the thinning property for Poisson processes that
 $$ |\xi_t| \to \left\{\begin{array}{rcl} |\xi_t| + 1 & \hbox{at rate} & \lambda e^{a |\xi_t|/N} |\xi_t| (1 - |\xi_t|/N), \vspace*{4pt} \\
                                          |\xi_t| - 1 & \hbox{at rate} & |\xi_t|. \end{array} \right. $$
 Rescaling population to the proportion of occupied sites~$u_t = |\xi_t|/N$, we get
 $$ u_t \to \left\{\begin{array}{rcl} u_t + 1/N & \hbox{at rate} & N \lambda e^{a u_t} u_t (1 - u_t), \vspace*{4pt} \\
                                      u_t - 1/N & \hbox{at rate} & N u_t. \end{array} \right. $$ 
 It follows that the mean-field model is described by the differential equation
 $$ \begin{array}{rcl}
      u' & \n = \n & \lim_{N \to \infty} \,N \ \E [u_{t + 1/N} - u_t \,| \,u_t = u] \vspace*{4pt} \\
         & \n = \n & \lim_{N \to \infty} \,N \,((1/N) \,\p [u_{t + 1/N} = u_t + 1/N \,| \,u_t = u] \vspace*{4pt} \\
         &         & \hspace*{75pt} - \ (1/N) \,\p [u_{t + 1/N} = u_t - 1/N \,| \,u_t = u] + o (1/N)) \vspace*{4pt} \\
         & \n = \n & \lambda e^{au} u (1 - u) - u = \phi (u). \end{array} $$
 Because of the exponential form of the birth rate, we cannot obtain the exact expression of the fixed points.
\begin{figure}[t!]
\centering
\scalebox{0.40}{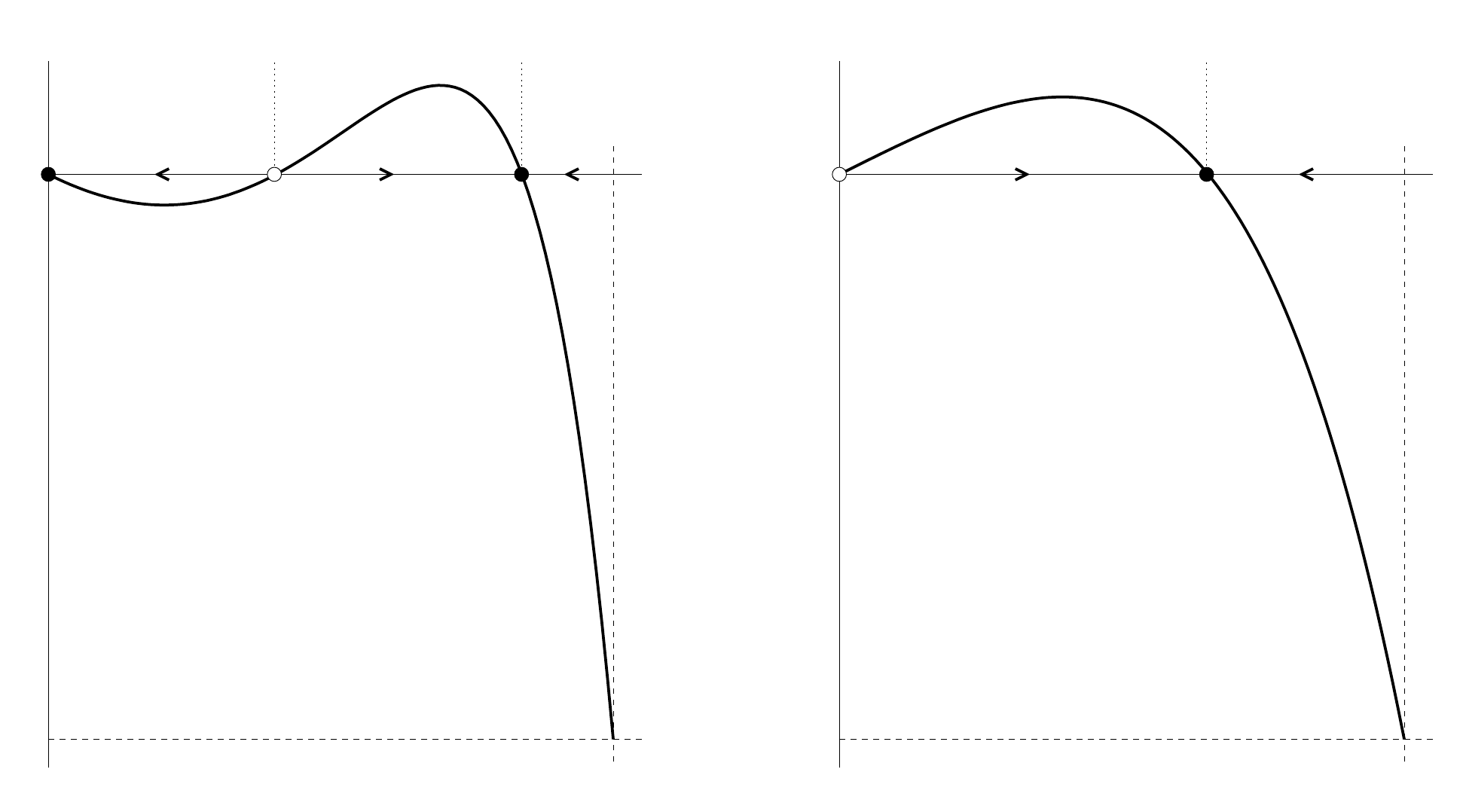}
\caption{\upshape{
 Dynamical structure of the mean-field model.
 The left picture shows an example where~$\lambda < 1$ and~$a > a_c (\lambda)$, leading to bistability, while the right picture shows an example where~$\lambda > 1$, leading to a globally stable interior fixed point.}}
\label{fig:phi}
\end{figure}
 However, the stability of the trivial fixed point~0, corresponding to the extinction state, as well as the existence and stability of additional~(interior) fixed points, can be studied.
 To begin with, observe that, for all fixed~$- \infty < a < \infty$,
\begin{equation}
\label{eq:stability}
\begin{array}{rcl} u \approx 0 & \Longrightarrow & \phi (u) \approx \lambda u - u = (\lambda - 1) u. \end{array}
\end{equation}
 This shows that the stability of the trivial fixed point, and so whether the population survives or dies out when starting at low density, depends on the birth rate~$\lambda$ but not on the payoff~$a$, which is in sharp contrast with Theorems~\ref{th:survival} and~\ref{th:extinction}. \vspace*{5pt} \\
\noindent
{\bf Extinction phase.}
 It follows from~\eqref{eq:stability} that, when~$\lambda < 1$, the trivial fixed point is locally stable, so the population dies out~($u \to 0$ starting at low density) even when the payoff~$a$ is very large positive.
 This contrasts with Theorem~\ref{th:survival}, which states that, even starting with a finite number of players~(density zero), the population survives with positive probability.
 However, for all~$\lambda > 0$ and~$\bar u < 1/2$,
 $$ \begin{array}{rcl} a > \ln (2 / \lambda) / \bar u & \Longrightarrow & \phi (\bar u) > \lambda e^{\ln (2 / \lambda)} \,\bar u (1 - \bar u) - \bar u = 2 \bar u (1 - \bar u) - \bar u \geq 0. \end{array} $$
 This shows that~$\phi (u)$ becomes positive at some unstable fixed point~$u_- \in (0, \bar u)$, and so that the system is bistable.
 In particular, starting at a density~$> u_-$, the population converges to a limit~$u_+ > u_- > 0$.
 In conclusion, when the birth rate~$\lambda < 1$, there is extinction in the sense that the trivial fixed point is locally stable, but the basin of attraction of~0 can be made arbitrarily small by taking the payoff~$a$ sufficiently large positive.
 See Figure~\ref{fig:phi} for a picture.
 In fact, more can be said about the parameter region where the system is bistable.
 Note that bistability appears when there is a unique interior fixed point~$u_0$ such that the graph of~$\phi$ is tangent to the $x$-axis at point~$u_0$, which means that both~$\phi (u_0)$ and~$\phi' (u_0)$ are equal to zero.
 In particular,~$u_0 > 0$ and
\begin{equation}
\label{eq:bistable-1}
\phi (u_0) = \lambda e^{au_0} u_0 (1 - u_0) - u_0 = 0 \quad \hbox{and so} \quad \lambda e^{au_0} (1 - u_0) = 1.
\end{equation}
 In addition, looking at the derivative, we get
\begin{equation}
\label{eq:bistable-2}
\begin{array}{rcl}
\phi' (u_0) & \n = \n & \lambda a e^{au_0} u_0 (1 - u_0) + \lambda e^{au_0} (1 - u_0) - \lambda e^{au_0} u_0 - 1 \vspace*{4pt} \\
            & \n = \n & \lambda e^{au_0} (au_0 (1 - u_0) - 2u_0 + 1) - 1 \vspace*{4pt} \\
            & \n = \n & \lambda e^{au_0} (1 - u_0)(au_0 + 2) - \lambda e^{au_0} - 1 = 0.  \end{array}
\end{equation}
 Using~\eqref{eq:bistable-1}, we can replace~$\lambda e^{au_0} (1 - u_0)$ with~1 in~\eqref{eq:bistable-2} to get
\begin{equation}
\label{eq:bistable-3}
  (au_0 + 2) - \lambda e^{au_0} - 1 = 0 \quad \hbox{and so} \quad \lambda e^{au_0} = 1 + au_0.
\end{equation}
 Now, for all~$\lambda < 1$, let~$x_{\lambda}$ be the unique positive solution of~$\lambda e^x = 1 + x$.
\begin{figure}[t!]
\centering
\scalebox{0.40}{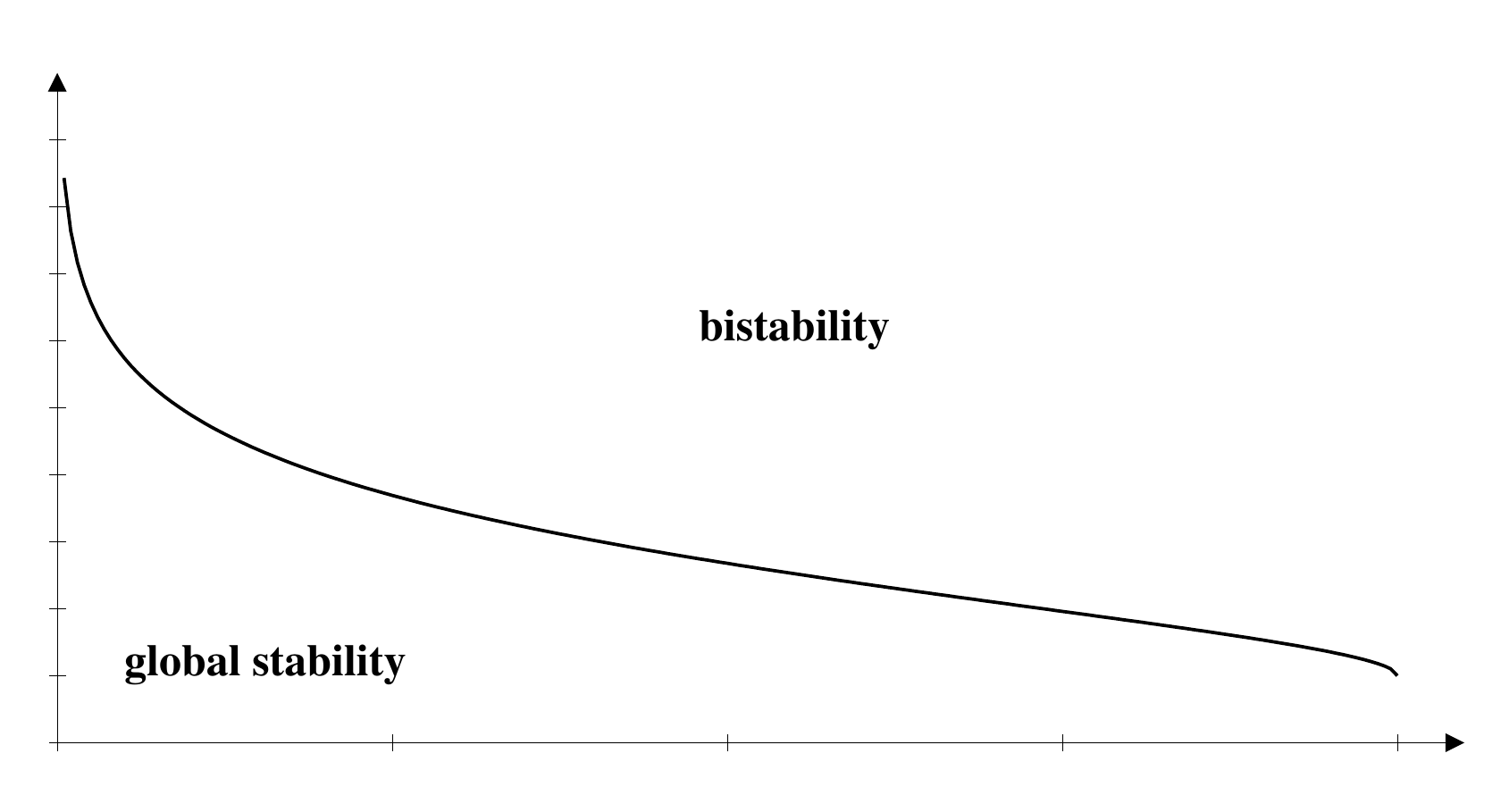}
\caption{\upshape{
 Critical payoff~$a_c (\lambda)$ for all birth rates~$\lambda \in (0, 1]$.
 For all parameter pairs~$(\lambda, a)$ above the graph, the mean-field model is bistable, whereas for all parameter pairs below the graph, the trivial fixed point~0 is globally stable.}}
\label{fig:bistable}
\end{figure}
 It follows from~\eqref{eq:bistable-3} that~$au_0 = x_{\lambda}$ so, using~\eqref{eq:bistable-1}, we deduce that
 $$ \lambda e^{x_{\lambda}} (1 - u_0) = 1 \quad \hbox{and so} \quad u_0 = 1 - e^{- x_{\lambda}} / \lambda. $$
 Using again that~$au_0 = x_{\lambda}$, we conclude that, for all birth rates~$\lambda < 1$, the system is bistable for all payoffs~$a$ larger than the critical payoff
 $$ a_c (\lambda) = \frac{x_{\lambda}}{u_0} = \frac{x_{\lambda}}{1 - e^{- x_{\lambda}} / \lambda} = \frac{x_{\lambda}}{1 - \frac{1}{1 + x_{\lambda}}} = 1 + x_{\lambda}, $$
  where again~$x_{\lambda}$ is the unique positive solution of~$\lambda e^x = 1 + x$.
  Figure~\ref{fig:bistable} shows a picture of the critical payoff~$a_c (\lambda)$ using this characterization. \vspace*{5pt} \\
{\bf Survival phase.}
 It follows from~\eqref{eq:stability} that, when~$\lambda > 1$, the trivial fixed point is unstable, so the population survives~(starting from a positive density,~$u$ converges to a positive limit) even when~$a$ is very large negative.
 This contrasts with Theorem~\ref{th:extinction}, which states that, regardless of the initial configuration~(even density one), the density of occupied sites vanishes to zero.
 However, for all~$\lambda < \infty$ and~$\bar u < 1/2$,
 $$ \begin{array}{rcl} a < \ln (1 / \lambda) / \bar u & \Longrightarrow & \phi (\bar u) < \lambda e^{\ln (1 / \lambda)} \,\bar u (1 - \bar u) - \bar u = \bar u (1 - \bar u) - \bar u \leq 0. \end{array} $$
 This shows that~$\phi (u)$ becomes negative at some stable fixed point~$u_+ \in (0, \bar u)$ therefore, starting at low density, the population converges to a limit that cannot exceeds~$\bar u$.
 In conclusion, when~$\lambda > 1$, there is survival in the sense that the trivial fixed point is unstable, but the limiting density can be made arbitrarily small by taking the payoff~$a$ sufficiently large negative.
 See Figure~\ref{fig:phi} for a picture.


\section{Proof of Theorem \ref{th:coupling} (monotonicity and attractiveness)}
\label{sec:coupling}
\noindent
 The proofs of Theorem~\ref{th:coupling} and~\eqref{eq:monotone} are based on coupling arguments.
 Processes with different parameters and/or initial configurations can be coupled by constructing them jointly on the same graphical representation~\cite{harris_1978}.
 However, the graphical representation of the contact process with a density-dependent birth rate is somewhat complicated, so we use instead the classical comparison result~\cite[Theorem III.1.5]{liggett_1985}, which we restate for the reader's convenience.
 From now on, we let~$c_{i \to j} (x, \xi)$ denote the rate at which site~$x$ jumps from state~$i$ to state~$j$ given that the configuration is~$\xi$.
\begin{theorem}
\label{th:comparison}
 Let~$\xi_t^1$ and~$\xi_t^2$ be two particle systems with state space~$\Omega = \{0, 1 \}^{\Z^d}$, and assume that, whenever~$\xi^1 \subset \xi^2$, we have the inequalities
 $$ c_{0 \to 1} (x, \xi^1) \leq c_{0 \to 1} (x, \xi^2) \quad \hbox{and} \quad c_{1 \to 0} (x, \xi^1) \geq c_{1 \to 0} (x, \xi^2). $$
 Then, there is a coupling of the two processes such that
 $$ \begin{array}{rcl} \xi^1 \subset \xi^2 & \Longrightarrow & \p^{(\xi^1, \xi^2)} [\xi_t^1 \subset \xi_t^2 \ \forall t] = 1. \end{array} $$
\end{theorem}
\noindent
 Using this result, we can prove the monotonicity and attractiveness of the contact process with a density-dependent birth rate, which implies~\eqref{eq:monotone}.
\begin{proof}[Proof of~\eqref{eq:monotone}]
 For~$i = 1, 2$, let~$\xi_t^i$ be the process with natural birth rate~$\lambda_i$, payoff coefficient~$a_i$, and initial configuration~$\xi^i$.
 Because the death rate is always equal to one, the second inequality in Theorem~\ref{th:comparison}, which is in fact an equality, is always satisfied.
 To prove the first inequality, recall that the birth rate of the process is of the form
 $$ \begin{array}{c} c_{0 \to 1} (x, \xi) = \sum_{y \sim x} \Phi (y, \xi) \,\xi (y) / 2d = \sum_{y \sim x} \lambda \exp (a f_1 (y, \xi)) \,\xi (y) / 2d = \psi (\lambda, a, \xi). \end{array} $$
 The function~$\psi (\lambda, a, \xi)$ is nondecreasing with respect to~$\lambda$,~$a$, and~$\xi$ when~$a \geq 0$.
 In particular, according to Theorem~\ref{th:comparison}, there is a coupling such that
\begin{equation}
\label{eq:monotone-1}
\begin{array}{rcl}
\lambda_1 \leq \lambda_2, \ \ 0 \leq a_1 \leq a_2, \ \ \xi^1 \subset \xi^2 & \Longrightarrow & \p^{(\xi^1, \xi^2)} [\xi_t^1 \subset \xi_t^2 \ \forall t] = 1. \end{array}
\end{equation}
 Note also that, for all~$\lambda_1 \leq \lambda_2$, $a_1 \leq 0 \leq a_2$, and~$\xi^1 \subset \xi^2$,
 $$ \begin{array}{rcl}
    \psi (\lambda_1, a_1, \xi_1) \leq \sum_{y \sim x} \lambda_1 \,\xi^1 (y) / 2d \leq \sum_{y \sim x} \lambda_2 \,\xi^2 (y) / 2d \leq \psi (\lambda_2, a_2, \xi_2),
    \end{array} $$
 from which we deduce that
\begin{equation}
\label{eq:monotone-2}
\begin{array}{rcl}
\lambda_1 \leq \lambda_2, \ \ a_1 \leq 0 \leq a_2, \ \ \xi^1 \subset \xi^2 & \Longrightarrow & \p^{(\xi^1, \xi^2)} [\xi_t^1 \subset \xi_t^2 \ \forall t] = 1. \end{array}
\end{equation}
 The implication in~\eqref{eq:monotone} directly follows from~\eqref{eq:monotone-1} and~\eqref{eq:monotone-2}.
\end{proof}
\noindent
 Using again Theorem~\ref{th:comparison}, we can also prove Theorem~\ref{th:coupling}.
\begin{proof}[Proof of Theorem~\ref{th:coupling}]
 Let
 $$ \begin{array}{rcl}
     \eta_t & \n = \n & \hbox{contact process with parameter} \ \lambda, \vspace*{4pt} \\
    \zeta_t & \n = \n & \hbox{contact process with parameter} \ \lambda e^{a (1 - 1/2d)}, \end{array} $$
 and let~$\xi_t$ be the contact process with a density-dependent birth rate.
 Observing that the contact process with a density-dependent birth rate and payoff~$a = 0$ reduces to the basic contact process with the same natural birth rate, it follows from~\eqref{eq:monotone-2} that the process~$\xi_t$ dominates~$\eta_t$ when~$a \geq 0$ but is dominated by~$\eta_t$ when~$a \leq 0$.
 Now, to compare the processes~$\xi_t$ and~$\zeta_t$, observe that, when site~$x$ is empty, its neighbors have at most~$2d - 1$ occupied neighbors, therefore
 $$ \begin{array}{rclcl}
      a \leq 0, \ \ \xi \supset \zeta & \Longrightarrow &
      c_{0 \to 1} (x, \xi) & \n \geq \n & \sum_{y \sim x} \lambda e^{a (1 - 1/2d)} \,\xi (y) / 2d = \lambda e^{a (1 - 1/2d)} \,f_1 (x, \xi) \vspace*{8pt} \\ &&
                           & \n \geq \n & \lambda e^{a (1 - 1/2d)} \,f_1 (x, \zeta) = c_{0 \to 1} (x, \zeta), \vspace*{8pt} \\
      a \geq 0, \ \ \xi \subset \zeta & \Longrightarrow &
      c_{0 \to 1} (x, \xi) & \n \leq \n & \sum_{y \sim x} \lambda e^{a (1 - 1/2d)} \,\xi (y) / 2d = \lambda e^{a (1 - 1/2d)} \,f_1 (x, \xi) \vspace*{8pt} \\ &&
                           & \n \leq \n & \lambda e^{a (1 - 1/2d)} \,f_1 (x, \zeta) = c_{0 \to 1} (x, \zeta). \end{array} $$
 In particular, according to Theorem~\ref{th:comparison}, the process~$\xi_t$ dominates~$\zeta_t$ when~$a \leq 0$ but is dominated by~$\zeta_t$ when~$a \geq 0$.
 In conclusion, there are couplings such that
 $$ \begin{array}{rcl}
      a \leq 0, \ \ \eta \supset \xi \supset \zeta & \Longrightarrow & \p^{(\eta, \xi, \zeta)} [\eta_t \supset \xi_t \supset \zeta_t \ \forall t] = 1, \vspace*{4pt} \\
      a \geq 0, \ \ \eta \subset \xi \subset \zeta & \Longrightarrow & \p^{(\eta, \xi, \zeta)} [\eta_t \subset \xi_t \subset \zeta_t \ \forall t] = 1. \end{array} $$
 Because~\cite[Theorem~1]{bezuidenhout_grimmett_1990} implies that~$\eta_t$ survives if and only if~$\lambda > \lambda_c$, and that~$\zeta_t$ survives if and only if~$\lambda > \lambda_c \,e^{- a (1 - 1/2d)}$, the theorem follows.
\end{proof}


\section{Proof of Theorem \ref{th:survival} (survival for~$a$ large positive)}
\label{sec:survival}
\noindent
 This section is devoted to the proof of Theorem~\ref{th:survival}, which states that, for all birth rates~$\lambda > 0$ even small, the process survives provided the payoff coefficient~$a$ is sufficiently large positive.
 Indeed, though it might be likely that isolated players die before they have a chance to give birth, players with at least one neighbor have a large birth rate when~$a$ is large positive, so clusters should expand.
 To turn the argument into a rigorous proof, we use a block construction.
 This technique first appeared in~\cite{bramson_durrett_1988} and is explained in detail in~\cite{durrett_1995}.
 The basic idea is to compare the process properly rescaled in space and time with supercritical oriented site percolation.
 Let
 $$ \Lat = \{(m, n) \in \Z^d \times \N : m_1 + \cdots + m_d + n \ \hbox{is even} \}, $$
 and turn~$\Lat$ into a directed graph~$\vec{\Lat}$ by placing an edge
 $$ (m, n) \to (m', n') \quad \hbox{if and only if} \quad m_i' = m_i \pm 1 \ \hbox{for all} \ i \ \hbox{and} \ n' = n + 1. $$
 Let~$\Lambda_- = \{0, 1 \}^d$ and~$\tau > 0$, and call~$(m, n) \in \Lat$ a good site if
 $$ E_{m, n} = \{\hbox{the box~$m + \Lambda_-$ is fully occupied at time~$n\tau$} \} \quad \hbox{occurs},$$
 where $m + \Lambda_- = \{m + x : x \in \Lambda_- \}.$
 The objective is to prove that, for all~$\ep > 0$, time~$\tau$ can be chosen in such a way that the set of good sites dominates the set of wet sites in an oriented site percolation process on~$\vec{\Lat}$ with parameter~$1 - \ep$.
 To do this, by translation invariance of the process in space and time, it suffices to prove that
\begin{equation}
\label{eq:double}
\begin{array}{rcl}
\xi \supset \Lambda_- & \Longrightarrow & \p_{\lambda, a}^{\xi} [\xi_{\tau} \supset \Lambda_+] \geq 1 - \ep \quad \hbox{where} \quad \Lambda_+ = \{-1, 0, 1, 2 \}^d, \end{array}
\end{equation}
 regardless of the configuration of the process outside the cube~$\Lambda_+$.
 Since in addition our process is attractive when~$a \geq 0$ and the birth rate of players with at least one occupied neighbor is always larger than~$\lambda e^{a / 2d}$, it suffices to prove the implication~\eqref{eq:double} for the new process~$\bar \xi_t$ with transition rates
 $$ \begin{array}{c} c_{0 \to 1} (x, \bar \xi_t) = \sum_{y \sim x} \lambda e^{a / 2d} \,\ind \{f_1 (y, \bar \xi_t) \neq 0 \} \,\bar \xi_t (y) / 2d \quad \hbox{and} \quad c_{1 \to 0} (x, \bar \xi_t) = 1, \end{array} $$
 modified so that births outside~$\Lambda_+$ are suppressed.
 The transition rates indicate that isolated players cannot give birth, while players with at least one occupied neighbor give birth at rate~$\lambda e^{a / 2d}$.
 This process can be constructed graphically as follows:
\begin{itemize}
 \item {\bf Births}.
       Equip each directed edge of the lattice~$\vec{xy}$, $x \sim y$, with an exponential clock with rate~$\lambda e^{a / 2d} / 2d$.
       At the times~$t$ the clock rings, draw an arrow~$(x, t) \to (y, t)$ to indicate that, if site~$x$ is occupied and has at least one occupied neighbor, and site~$y$ is empty, then site~$y$ becomes occupied. \vspace*{4pt}
 \item {\bf Deaths}.
       Equip each vertex~$x$ of the lattice with an exponential clock with rate one.
       At the times~$t$ the clock rings, put a cross~$\times$ at~$(x, t)$ to indicate that, if site~$x$ is occupied, then it becomes empty.
\end{itemize}
 We denote by~$\bar \p_{\lambda, a}$ the law of this process.
 To prove survival, we first show that, with high probability, there are no death marks~$\times$ in~$\Lambda_+$ by some small time~$\tau$.
\begin{lemma}
\label{lem:death}
 For all~$\ep > 0$, there exists~$\tau > 0$ such that
 $$ \bar \p_{\lambda, a} [\hbox{no death marks~$\times$ in~$\Lambda_+ \times [0, \tau]$}] \geq 1 - \ep/2. $$
\end{lemma}
\begin{proof}
 The number~$D_{\tau}$ of death marks in the space-time block~$\Lambda_+ \times [0, \tau]$ is Poisson distributed with parameter~$4^d \tau$, so the probability of no death marks satisfies
\begin{equation}
\label{eq:death}
\begin{array}{rcl} \p [D_{\tau} = 0] = e^{- 4^d \tau} = 1 - \ep/2 & \Longleftrightarrow &\tau = - \ln (1 - \ep / 2) / 4^d > 0. \end{array}
\end{equation}
 This completes the proof.
\end{proof}
\noindent
 The next step is to prove invasion~$\Lambda_- \to \Lambda_+$ by time~$\tau$.
 Because the players in~$\Lambda_-$ cannot immediately give birth onto the corners of~$\Lambda_+$ when $d > 1$, we divide the problem into~$d$ steps by considering the sequence of spatial regions
 $$ \begin{array}{c} \Lambda_i = \{y \in \Lambda_+ : \min_{x \in \Lambda_-} \norm{x - y}_1 = i \} \quad \hbox{for} \quad i = 0, 1, \ldots, d. \end{array} $$
 Figure~\ref{fig:Lambda} shows a picture in the~$d = 2$ case.
 Note that
 $$ \Lambda_- = \Lambda_0 \quad \hbox{and} \quad \Lambda_+ = \Lambda_0 \cup \Lambda_1 \cup \cdots \cup \Lambda_d. $$
 In addition, each site in one region has at least one neighbor in the previous region, therefore a fully occupied region can immediately invade the next region.
 We are now ready to prove invasion with probability close to one.
\begin{figure}[t!]
\centering
\scalebox{0.40}{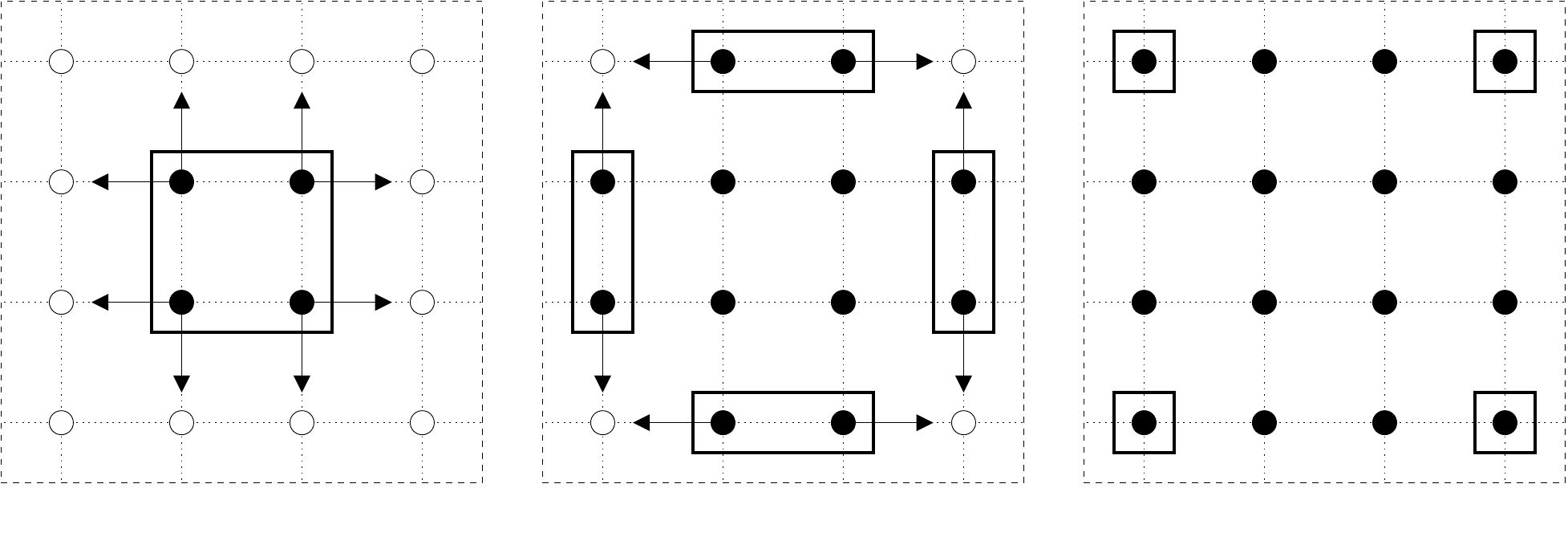}
\caption{\upshape{
 Illustration of the~$d$-step process used in the proof of Theorem~\ref{th:survival}.
 From left to right, the region delimited with bold lines represents the sets~$\Lambda_0$, $\Lambda_1$, and~$\Lambda_2$, respectively.
 In all three pictures, the dashed box represents~$\Lambda_+$.}}
\label{fig:Lambda}
\end{figure}
\begin{lemma}
\label{lem:birth}
 For all~$\ep > 0$ and $\tau > 0$ like in~\eqref{eq:death}, there exists~$a_+ < \infty$ such that
 $$ \bar \p_{\lambda, a}^{\Lambda_-} [\bar \xi_{\tau} = \Lambda_+] \geq 1 - \ep \quad \hbox{for all} \quad a \geq a_+. $$
\end{lemma}
\begin{proof}
 As previously mentioned, we proceed in~$d$ steps, and prove that, conditional on no deaths, each invasion~$\Lambda_i \to \Lambda_{i + 1}$ occurs in less then~$\tau / d$ units of time with high probability.
 Given that there are no deaths by time~$\tau$, and that~$\Lambda_i$ is fully occupied, the probability that~$\Lambda_{i + 1}$ becomes fully occupied in less than~$\tau / d$ units of time is larger than the probability that, for each site~$y \in \Lambda_{i + 1}$, there is at least one birth arrow~$x \to y$ for some~$x \in \Lambda_i$ in this time window.
 Because each region has less than~$4^d$ sites, and birth arrows occur along each directed edge at rate~$\lambda e^{a / 2d} / 2d$, this is larger than the probability that~$4^d$ independent exponential random variables~$X_1, X_2, \ldots, X_{4^d}$ with that rate are all less than~$\tau / d$.
 In particular, for~$i = 0, 1, \ldots, d - 1$,
\begin{equation}
\label{eq:birth}
\begin{array}{l}
\bar \p_{\lambda, a} [\bar \xi_t \supset \Lambda_{i + 1} \ \forall t \in [(i + 1) \tau / d, \tau] \,| \,\bar \xi_{i \tau / d} = \Lambda_i, D_{\tau} = 0] \vspace*{6pt} \\ \hspace*{34pt} =
\bar \p_{\lambda, a}^{\Lambda_i} [\bar \xi_{\tau / d} \supset \Lambda_{i + 1} \,| \,D_{\tau} = 0] \geq
\p [X_1, X_2, \ldots, X_{4^d} \leq \tau / d] \vspace*{6pt} \\ \hspace*{34pt} =
   (1 - \exp (- \lambda \tau e^{a / 2d} / 2d^2))^{4^d}. \end{array}
\end{equation}
 Because the right-hand side increases to one as~$a$ goes to infinity, there exists~$a_+$ large positive but finite such that the probability above is~$\geq 1 - \ep / 2d$ for all payoffs larger than~$a_+$.
 In particular, it follows from~\eqref{eq:birth} and Lemma~\ref{lem:death} that
 $$ \begin{array}{rcl}
    \bar \p_{\lambda, a}^{\Lambda_-} [\bar \xi_{\tau} \neq \Lambda_+] & \n \leq \n &
    \bar \p_{\lambda, a}^{\Lambda_-} [\bar \xi_{\tau} \neq \Lambda_+ \,| \,D_{\tau} = 0] + \p [D_{\tau} \neq 0] \vspace*{6pt} \\ & \n \leq \n &
    \sum_{i < d} \bar \p_{\lambda, a}^{\Lambda_i} [\bar \xi_{\tau / d} \not \supset \Lambda_{i + 1} \,| \,D_{\tau} = 0] + \p [D_{\tau} \neq 0] \vspace*{6pt} \\ & \n \leq \n &
    d (1 - (1 - \ep / 2d)) + \ep / 2 = \ep \end{array} $$
 for all~$a \geq a_+$.
 This proves the lemma.
\end{proof}
\noindent
 Using the lemma, we are now ready for the block construction.
\begin{lemma}
\label{lem:perco}
 Let~$\tau$ as in Lemma~\ref{lem:death} and~$a_+$ as in Lemma~\ref{lem:birth}.
 Then, for all~$a \geq a_+$, the set of good sites dominates stochastically the set of open sites in an oriented site percolation process on~$\vec{\Lat}$ with parameter~$1 - \ep$.
\end{lemma}
\begin{proof}
 Because~$\xi_t$ dominates~$\bar \xi_t$, and the evolution rules of the process~$\xi_t$ are translation invariant, it follows from Lemma~\ref{lem:birth} that, for all sites~$(m', n') \leftarrow (m, n)$,
 $$ \begin{array}{l}
    \p_{\lambda, a} [\xi_{(n + 1) \tau} \supset m' + \Lambda_- \,| \,\xi_{n \tau} \supset m + \Lambda_-] \vspace*{4pt} \\ \hspace*{50pt} \geq
    \p_{\lambda, a} [\xi_{(n + 1) \tau} \supset m + \Lambda_+ \,| \,\xi_{n \tau} \supset m + \Lambda_-] \geq 1 - \ep \quad \hbox{for all} \quad a \geq a_+. \end{array} $$
 This shows the existence of a collection of good events~$G_{m, n}$ such that
 $$ \hbox{(a)} \ \p_{\lambda, a} [G_{m, n}] \geq 1 - \ep \quad \hbox{and} \quad \hbox{(b)} \ E_{m, n} \cap G_{m, n}\subset E_{m', n'} \ \ \hbox{for all} \ \ (m', n') \leftarrow (m, n). $$
 The lemma then follows from~\cite[Theorem~A.4]{durrett_1995}
\end{proof}
\noindent
 To deduce the theorem, start~$\xi_t$ with a single player at the origin.
 Because~$\lambda > 0$, there is a positive probability that the small box~$\Lambda_-$ is fully occupied at time one, so we may assume without loss of generality that the process starts from~$\xi_0 = \Lambda_-$ instead.
 In addition, because Lemma~\ref{lem:birth} applies to the modified process~$\bar \xi_t$, the events~$G_{m, n}$ in Lemma~\ref{lem:perco} can be made measurable with respect to the graphical representation in a bounded space-time block.
 This shows that the range of dependence of the oriented site percolation process is finite, so we can fix~$\ep > 0$ small to make the percolation process supercritical.
 Since in addition, according to Lemma~\ref{lem:perco}, the set of good sites dominates the set of wet sites for all~$a \geq a_+$, the process survives.


\section{Proof of Theorem \ref{th:extinction} (extinction for~$a$ large negative)}
\label{sec:extinction}
\noindent
 This section is devoted to the proof of Theorem~\ref{th:extinction}, which states that, for all~$\lambda < \infty$ even large, the process dies out if the payoff~$a$ is sufficiently large negative.
 Like in Theorem~\ref{th:survival}, the basic idea is that isolated players and players with at least one neighbor may have fitnesses that differ strongly, except that the effects are now reversed:
 though it might be likely that isolated players give birth quickly, once they give birth and have one neighbor~(their offspring), they are no longer likely to give birth when~$a$ is large negative.
 The proof again relies on a block construction, but the technical details are somewhat more complicated.
 To begin with, we consider the process starting with a single individual in the limit~$a = - \infty$.
 In this case, adjacent individuals cannot give birth and the process essentially behaves like a symmetric random walk that dies after a geometric number of jumps.
 This implies that the length of the invasion path decays exponentially in both space and time~(radius and time to extinction).
 Because the birth rate decreases with the local density of individuals, the process starting from a general configuration is dominated stochastically by a system of such independent random walks starting with one particle per site.
 The exponential decay implies that, regardless of the configuration outside a large space-time block and with probability close to one, a large space-time region around the center of the block is not reached by any of the random walks/invasion paths~(dead region).
 This is used to prove percolation of the dead regions under a suitable space-time rescaling.
 Once the space and time scales are fixed, we can use a perturbation argument to prove that percolation still occurs for all~$a$ sufficiently large negative.
 Extinction of the process then follows from the fact that the percolation parameter can be chosen close enough to one to ensure the lack of percolation of the closed sites~(corresponding to potentially occupied blocks), and the fact that individuals in our process cannot appear spontaneously. \\


\noindent{\bf Fast extinction for $a = - \infty$.}
 We first study the limiting case~$a = - \infty$ starting with a single player at the origin, whose law is denoted by~$\p_{\lambda, - \infty}^0$. Let
 $$ T = \inf \{t : \xi_t = \varnothing \} = \hbox{time to extinction}. $$
\begin{lemma}
\label{lem:exp-decay-time}
 For every $\lambda \geq 0$, there exists a constant $\beta = \beta(\lambda,d) > 0$ such that 
 $$ \p_{\lambda,-\infty}^0 [T \geq t] \leq e^{-\beta t} \quad \hbox{for all} \quad t \geq 0. $$
\end{lemma}
\begin{proof}
 Because~$a = -\infty$, two adjacent players cannot give birth, so there can be at most two players alive at the same time.
 By the superposition property, the time until one of the two players dies is distributed as an exponential random variable with parameter two.
 In addition, if there is just a single player, the probability that it dies before it reproduces is equal to~$1 / (1 + \lambda)$.
 Therefore, the total number of generations~$N$ in which exactly two players are alive is geometrically distributed:
\begin{equation}
\label{eq:exp-decay-time-1}
\p_{\lambda, -\infty}^0 [N > n] \leq \left(1 - \frac{1}{1 + \lambda} \right)^n = \left(\frac{\lambda}{1 + \lambda} \right)^n = \bigg(1 + \frac{1}{\lambda} \bigg)^{-n}, \quad n \in \N.
\end{equation}
 Note also that the time increments between consecutive generations can be written rigorously using the sequence of stopping times~$\tau_0, \tau_1, \ldots, \tau_{2N + 1}$, starting at~$\tau_0 = 0$, defined recursively as
 $$ \begin{array}{rclcl}
    \tau_{2i + 1} & \n = \n & \inf \{t > \tau_{2i} : \hbox{0 or 2 players alive at time} \ t \} & \hbox{for} & i = 0, 1, \ldots, N, \vspace*{4pt} \\
    \tau_{2i + 2} & \n = \n & \inf \{t > \tau_{2i + 1} : \hbox{1 player alive at time} \ t \}   & \hbox{for} & i = 0, 1, \ldots, N - 1.
 \end{array} $$
 In particular, the time to extinction~$T = \tau_{2N + 1}$ of the process can be written as a random sum of independent exponential random variables:
 $$ T = \tau_{2N + 1} = \tau_1 + \sum_{i = 1}^{2N} \ (\tau_{i + 1} - \tau_i) \stackrel{\text{d}}{=} S_0 + \sum_{i = 1}^N \ (T_i + S_i), $$
 where the~$(T_i)_{i \in \N}$ are iid exponential random variables with parameter two, the~$(S_i)_{i \in \N}$ are iid exponential random variables with parameter~$(1+\lambda)$ and the sum is interpreted as empty if $N = 0$.
 In particular,~$T$ is a hypoexponential random variable.
 Because we are interest in the case where the birth parameter~$\lambda$ is large~(and so typically~$\lambda > 1$), and we do not need very precise bounds, we can upper bound the time~$T$ to extinction by the sum~$\Tilde T_N$ of~$2N + 1$ independent exponential random variables with the same parameter one.
 Then, the random variable~$\Tilde T_N$ is a~$\text{Gamma} (2N + 1, 1)$.
 For a deterministic~$n$, one can use the following Chernoff bound for the tails of the Gamma distribution:
\begin{equation}
\label{eq:exp-decay-time-2}
\p_{\lambda, -\infty}^0 [\Tilde T_n \geq t] \leq \frac{e^{-\theta t}}{(1 - \theta)^{2n + 1}}.
\end{equation}
 Combining~\eqref{eq:exp-decay-time-1} and~\eqref{eq:exp-decay-time-2} implies that, for any~$\theta < 1$ and~$m \in \N$,
 $$ \p_{\lambda,- \infty}^0 [T \geq t] \leq \sum_{n = 1}^m \ \p [\Tilde T_n \geq t] + \p [N > m]
                                       \leq m \bigg(\frac{e^{-\theta t}}{(1 - \theta)^{2m + 1}} \bigg) + e^{- m \ln (1 + 1 / \lambda)}. $$
 Taking for example~$\theta = 1/2$ and~$m = \lfloor tr \rfloor$ with~$r < 1 / 4 \ln (2)$ gives the result. 
\end{proof}
\noindent
 Looking closely at the proof of the previous lemma, one easily deduces the following result, which also gives us an exponential decay of the radius of the invasion paths.
\begin{lemma}
\label{lem:exp-decay-space}
 Let~$\xi^0 \subset \Z^d$ be the (random) subset of vertices that are ever occupied.
 Then, for every birth rate~$\lambda \geq 0$, there exists a constant~$\delta = \delta (\lambda, d) > 0$ such that
 $$ \p_{\lambda, - \infty}^0 [|\xi^0| > n] \leq e^{- \delta n} \quad \hbox{for all} \quad n \in \N. $$
\end{lemma}
\begin{proof}
 Recall from the proof of Lemma~\ref{lem:exp-decay-time} that the number of generations~$N$ for which our process survives is a geometric random variable with parameter~$1/( 1 + \lambda)$.
 Since we can only visit at most one new site per generation, it follows from~\eqref{eq:exp-decay-time-1} that
 $$ \p_{\lambda, -\infty}^0 [|\xi^0| > n] \leq \p [N > n] = (1 + 1 / \lambda)^{-n} = e^{- n \ln (1 + 1 / \lambda)}, $$
 which is exactly the exponential decay we were looking for.
\end{proof}


\noindent{\bf Comparison to non-interacting copies.}
 To use union bounds, it will be convenient to compare the process with arbitrary initial condition to a family of non-interacting copies starting with a single occupied site.
 For all~$A \subset \Z^d$, we let~$\xi_t^A$ denote the process starting from the set~$A$ occupied, and simply write~$\xi_t^z$ when~$A = \{z \}$.
 Having a collection~$\{(\xi^z_t)_{t \geq 0} : \ z \in A \}$ of non-interacting copies of our process, we let
 $$ \begin{array}{c} \Xi_t^A (x) = \sum_{z \in A} \xi_t^z (x) \quad \hbox{for all} \quad x \in \Z^d. \end{array} $$
 With this notation at hand, we are ready to state the following comparison result.
\begin{lemma}
\label{lem:non-interacting-comparison}
 There exists a coupling such that
 $$ \p_{\lambda, - \infty} [\xi^A_t (x) \leq \Xi^A_t (x) \ \forall (x, t) \in \Z^d \times \R_+] = 1. $$
\end{lemma}
\noindent
 In particular, the set of sites occupied by the original process at time $t$ is always contained in the set of sites occupied by at least one of the non-interacting copies.
\begin{proof}
 To compare the two processes, we differentiate the players assuming that, for all~$z \in A$, site~$z$ is initially occupied by a type~$z$ player.
 Then, in the limiting case~$a = - \infty$, we can construct all the processes from the following graphical representation:
\begin{itemize}
 \item {\bf Births}.
       For each~$z \in A$, equip each~$\vec{xy}$, $x \sim y$, with a rate~$\lambda / 2d$ exponential clock.
       At the times~$t$ the clock rings, draw an arrow~$(x, t) \overset{z}{\longrightarrow} (y, t)$. \vspace*{4pt}
 \item {\bf Deaths}.
       For each~$z \in A$, equip each~$x$ with a rate one exponential clock.
       At the times~$t$ the clock rings, put a cross~$\times_z$ at the space-time point~$(x, t)$.
\end{itemize}
 The crosses have the same effects on both processes: a cross~$\times_z$ at site~$x$ kills a type~$z$ particle at that site.
 The arrows, however, have different effects.
\begin{itemize}
\item
 The process~$\xi_t^A$ is constructed by assuming that, if the tail~$x$ of a type~$z$ arrow is occupied by a type~$z$ player, and none of the neighbors of~$x$ is occupied, then the head~$y$ of the arrow becomes occupied by a type~$z$ player. \vspace*{4pt}
\item
 In contrast, the system~$\Xi^A_t$ of non-interacting copies is constructed by assuming that, if the tail of a type~$z$ arrow is occupied by a type~$z$ player, and none of the neighbors of site~$x$ is occupied by a type~$z$ player, then the head~$y$ of the arrow becomes occupied by a type~$z$ player.
\end{itemize}
 Because the condition for giving birth in the first process is more restrictive~(no players of any type in the neighborhood as opposed to no type~$z$ players), if there is a type~$z$ player at~$(x, t)$ in the first process then there is a type~$z$ player at~$(x, t)$ in the second process, which proves the lemma.
 Note that there is at most one type~$z$ player at each site in both processes, but sites can be occupied by multiple players with different types in the system of non-interacting copies.
\end{proof}


\noindent{\bf Block construction.}
 Using the exponential decay~(in space and time) of the invasion paths and the previous stochastic domination, we can now use a block construction to prove extinction~(and more importantly control the rate of extinction) of the process.
 Let~$\Lat_d = \Z^d \times \N$, which we turn into a directed graph~$\vec{\Lat}_d$ by placing an edge
 $$ \begin{array}{l}
    (m, n) \to (m', n') \quad \hbox{if and only if} \vspace*{4pt} \\ \hspace*{35pt}
    |m_1 - m_1'| + \cdots + |m_d - m_d'| + |n - n'| = 1 \ \hbox{and} \ n \leq n'. \end{array} $$
 In other words, starting from each site~$(m, n)$, there are~$2d$ ``horizontal'' arrows that we can think of as potential invasions in space, and one ``vertical'' arrow that we can think of as a potential persistence in time.
 To rescale the interacting particle system in space and time, we let~$L$ be a large integer, and define the space-time blocks
 $$ \begin{array}{rcl}
      A_{m,n} & \n = \n & (2mL, nL) + [-2L, 2L]^d \times [0, 2L], \vspace*{4pt} \\
      B_{m,n} & \n = \n & (2mL, nL) + [-L, L]^d \times [L, 2L], \end{array} $$
 for all~$(m, n) \in \Lat_d$.
 See Figure~\ref{fig:block} for a picture.
 We call~$(m, n)$ a good site whenever
 $$ E_{m, n} = \{\hbox{the space-time block~$B_{m, n}$ is empty} \} \quad \hbox{occurs}. $$
 We now prove that, when~$a = - \infty$ and regardless of the configuration outside~$A_{m, n}$, the event~$E_{m, n}$ occurs with probability close to one when~$L$ is large.
\begin{figure}[t!]
\centering
\scalebox{0.34}{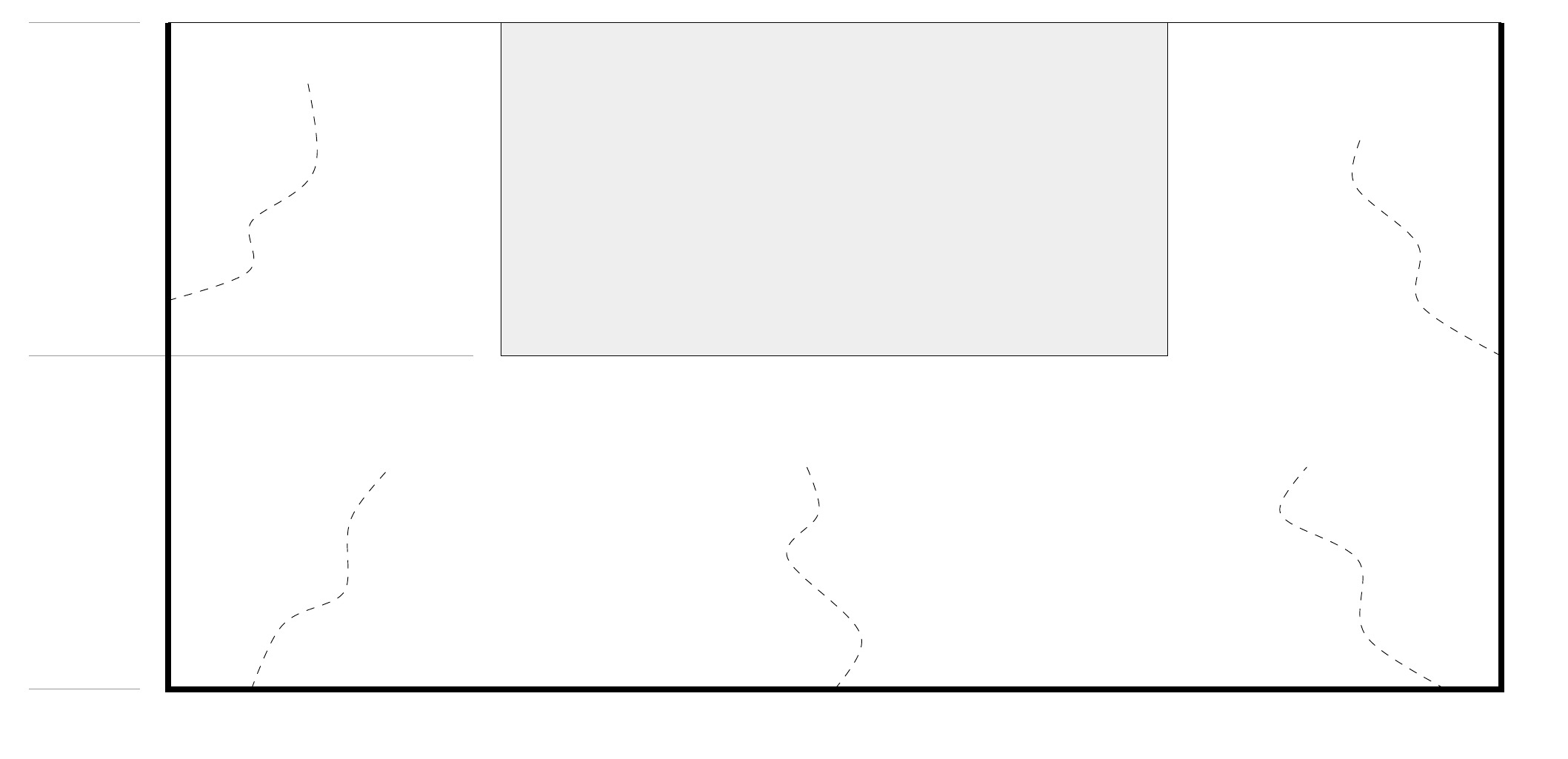}
\caption{\upshape{
 Illustration of the event~$E_{m, n}$ when~$d = 1$.
 With high probability, the invasion paths starting from the bottom or the periphery of the big space-time block~$A_{m, n}$ do not reach the smaller space-time block~$B_{m, n}$.}}
\label{fig:block}
\end{figure}
\begin{proposition}
\label{prop:good-blocks}
 Let $\ep > 0$.
 Then, for~$L \in \N$ sufficiently large,
 $$ \p_{\lambda, - \infty} [E_{m, n}] \geq 1 - \ep / 2. $$
\end{proposition}
\begin{proof}
 In view of the stochastic domination in Lemma~\ref{lem:non-interacting-comparison}, it suffices to prove the result for the process that evolves according to~$\xi_t$ outside the space-time block~$A_{m, n}$ but according to~$\Xi_t$ inside the block.
 The idea is to use Lemma~\ref{lem:exp-decay-time} to prove that the invasion paths starting from the bottom of the block cannot live too long, and Lemma~\ref{lem:exp-decay-space} to prove that the paths starting from the periphery cannot go too far.
 More precisely, let
 $$ \begin{array}{rcl}
    \Lambda_- & \n = \n & \{(x, t) \in A_{m, n} : t = nL \} = \hbox{bottom of~$A_{m, n}$}, \vspace*{4pt} \\
    \Lambda_+ & \n = \n & \{(x, t) \in A_{m, n} : \norm{x - 2mL}_{\infty} = 2L \} = \hbox{periphery of~$A_{m, n}$}, \end{array} $$
 and write~$\Lambda_{\pm} \to B_{m, n}$ to indicate an invasion path going from~$\Lambda_{\pm}$ to the block~$B_{m, n}$.
 Because players cannot appear spontaneously, we have
\begin{equation}
\label{eq:good-blocks-0}
\p_{\lambda, - \infty} [E_{m, n}^c] \leq \p_{\lambda, - \infty} [\Lambda_- \to B_{m, n}] + \p_{\lambda, - \infty} [\Lambda_+ \to B_{m, n}].
\end{equation}
 Now, by the domination in Lemma~\ref{lem:non-interacting-comparison} and the exponential decay in Lemma~\ref{lem:exp-decay-time},
\begin{equation}
\label{eq:good-blocks-1}
\p_{\lambda, - \infty} [\Lambda_- \to B_{m, n}] \leq |\Lambda_-| \times \p_{\lambda, - \infty} [\xi_L^0 \neq \varnothing] \leq (4L + 1)^d \,e^{- \beta L}.
\end{equation}
 Dealing with the invasion paths coming in the block~$A_{m, n}$ through its periphery~$\Lambda_+$ is more complicated because the number of such paths is random.
 Note that the number of paths is dictated by the process~$\xi_t$ whereas their length is upper bounded by their counterpart for the process~$\Xi_t$.
 To deal with the number of invasion paths coming in through the periphery of the space-time block, we let
 $$ H_{m, n} = \{\hbox{there are fewer than~$2e \lambda |\Lambda_+$| birth events onto $\Lambda_+$ from outside~$A_{m, n}$} \}, $$
 where~$|\Lambda_+|$ denotes the $d$-dimensional area of the union of the~$2d$ facets of the~$(d + 1)$-dimensional space-time block~$A_{m, n}$ whose normal vectors are orthogonal to the time axis~(see Figure~\ref{fig:block} for an illustration in $(1 + 1)$ dimensions).
 Because the births in the contact process~$\xi_t$ occur at rate at most~$\lambda$, the number of births onto~$\Lambda_+$ is dominated by the Poisson random variable with parameter~$\lambda |\Lambda_+|$.
 In particular, using the standard tail bound for the Poisson distribution, we get
 $$ \p_{\lambda, - \infty} [H_{m, n}^c] \leq 2 \exp (- \lambda |\Lambda_+| - 2e \lambda |\Lambda_+| ( \log (2) - 1)). $$
 This, together with Lemmas~\ref{lem:exp-decay-space} and~\ref{lem:non-interacting-comparison} and a union bound, implies that
\begin{equation}
\label{eq:good-blocks-2}
\begin{array}{l}
\p_{\lambda, - \infty} [\Lambda_+ \to B_{m,n}] \leq \p_{\lambda, - \infty} [\Lambda_+ \to B_{m, n} \,| \,H_{m, n}] + \p_{\lambda, - \infty} [H_{m, n}^c] \vspace*{4pt} \\ \hspace*{50pt}
                                               \leq  2e \lambda |\Lambda_+| \,\p_{\lambda, - \infty}^0 [|\xi^0| > L] + \p_{\lambda, - \infty} [H_{m, n}^c] \vspace*{4pt} \\ \hspace*{50pt}
                                               \leq  2e \lambda |\Lambda_+| \,e^{- \delta L} + 2 \exp (- \lambda |\Lambda_+| - 2e \lambda |\Lambda_+|(\log (2)-1)). \end{array}
\end{equation}
 Because~$|\Lambda_+| \leq 2d (4L + 1)^{d - 1} \times 2L$ we can choose~$L$ large to make both~\eqref{eq:good-blocks-1} and~\eqref{eq:good-blocks-2} arbitrarily small, hence plugging these estimates into~\eqref{eq:good-blocks-0} yields the claim.
\end{proof}


\noindent{\bf Perturbation argument.}
 It follows from the block construction in the previous subsection that the set of good sites~$(m, n) \in \Lat_d$~(indicating that the block~$B_{m, n}$ is empty) dominates the set of open sites in the percolation process on~$\vec{\Lat}_d$ with parameter~$1 - \ep$.
 Because~$\ep$ can be chosen arbitrarily small, in which case the set of closed~(not open) sites does not percolate, and because players cannot appear spontaneously, this shows extinction of the process.
 To complete the proof of Theorem~\ref{th:extinction}, the last step is to transport this result from the limiting case~$a = - \infty$ to the case where the payoff coefficient~$a$ is large negative but finite.
 In the latter case, an individual with a neighbor can give birth, but it is unlikely, so~(the graphical representations of) the two processes in a space-time block should agree with high probability, as long as one chooses~$a$ sufficiently large negative, depending on the size of the block.
 Instead of working with the graphical representation in Lemma~\ref{lem:non-interacting-comparison}, which was designed to compare our process with the system of non-interacting copies, we construct the processes with~$a > - \infty$ and~$a = - \infty$ using the following graphical representation:
\begin{itemize}
 \item {\bf Births}.
       For~$i = 0, 1, \ldots, 2d$, equip each~$\vec{xy}$, $x \sim y$, with a rate~$\lambda e^{i a/2d}/ 2d$ exponential clock.
       At the times~$t$ the clock rings, draw an arrow~$(x, t) \overset{i}{\longrightarrow} (y, t)$. \vspace*{4pt}
 \item {\bf Deaths}.
       Equip each~$x$ with a rate one exponential clock.
       At the times~$t$ the clock rings, put a cross~$\times$ at the space-time point~$(x, t)$.
\end{itemize}
 The crosses have the same effect on both processes: a cross~$\times$ at site~$x$ kills a player at that site.
 The process with~$a > - \infty$ is constructed by assuming that, if the tail~$x$ of a type~$i$ arrow is occupied, the head~$y$ is empty, and exactly~$i$ of the neighbors of~$x$ are occupied, then the head~$y$ of the arrow becomes occupied.
 The process with~$a = - \infty$ is constructed similarly but using only the type~0 arrows, since the other arrows occur at rate zero in the limit.
 Using this graphical representation, we can now extend Proposition~\ref{prop:good-blocks} to the process with~$a > - \infty$ sufficiently large negative.
\begin{lemma}
\label{lem:good-blocks}
 Let $\ep > 0$.
 Then, for~$L \in \N$ as in Proposition~\ref{prop:good-blocks},
 $$ \p_{\lambda, a} [E_{m, n}] \geq 1 - \ep \quad \hbox{for all} \quad a > - \infty \quad \hbox{sufficiently large negative}. $$
\end{lemma}
\begin{proof}
 In view of Proposition~\ref{prop:good-blocks}, it suffices to show that, with probability arbitrarily close to one, the process with~$a = - \infty$ and the process with~$a > - \infty$ large negative agree in the block~$A_{m, n}$, i.e., there are no type~$i$ arrows for~$i \neq 0$ in the block, which prevents players with at least one neighbor from giving birth.
 The overall rate of all the type~$i$ arrows for~$i \neq 0$ starting at a given site is
 $$ \lambda e^{a / 2d} + \lambda e^{2a / 2d} + \lambda e^{3a / 2d} + \cdots + \lambda e^{2da / 2d} \leq 2 \lambda e^{a / 2d} $$
 for all~$a > - \infty$ large negative, therefore the number of such arrows that point at the space-time block~$A_{m, n}$ is dominated by a Poisson random variable~$X$ with parameter
 $$ |A_{m, n}| \times 2 \lambda e^{a / 2d} = 2L (4L + 1)^d \times 2 \lambda e^{a / 2d}. $$
 Now, the scale parameter~$L \in \N$ being fixed so that Proposition~\ref{prop:good-blocks} holds, we can choose~$a_- > - \infty$ sufficiently large negative so that
\begin{equation}
\label{eq:match}
\p [X= 0] \geq 1 -\ep/2 \quad \hbox{for all} \quad a \leq a_-.
\end{equation}
 Finally, because the event~$E_{m, n}$ occurs for the process with~$a \leq a_-$ whenever it occurs for the process with~$a = - \infty$ and the two processes agree in the space-time block, it follows from~\eqref{eq:match} and Proposition~\ref{prop:good-blocks} that
 $$ 
    \p_{\lambda, a} [E_{m, n}]  \geq  \p_{\lambda, - \infty} [E_{m, n}] \times \p [X = 0]  \geq  (1 - \ep / 2)^2 \geq 1 - \ep
$$
 for all~$a \leq a_-$.
 This completes the proof.
\end{proof}
\noindent
 Using Lemma~\ref{lem:good-blocks} and an idea from~\cite{berg_grimmett_schinazi_1998}, we can now conclude the proof.
\begin{proof}[Proof of Theorem \ref{th:extinction}]
 It follows from Lemma~\ref{lem:good-blocks}~(and its proof) that there exists a collection of good events~$G_{m, n}$ that only depend on the graphical representation in the slightly enlarged space-time blocks
 $$ A^+_{m, n} = (2mL, nL) + [- 2L - 1, 2L + 1]^d \times  [0, 2L], \quad (m, n) \in \Lat_d, $$
 such that, for every~$\ep > 0$, we can choose the scale parameter~$L \in \N$ large, then the payoff coefficient~$a > - \infty$ large negative, to guarantee  that
 $$ \p_{\lambda, a} [G_{m, n}] \geq 1 - \ep \quad \hbox{and} \quad G_{m, n} \subset E_{m, n}. $$
 This implies that the set of good sites dominates stochastically the set of open sites in the oriented site percolation process on~$\vec \Lat_d$ with parameter $1 - \ep$.
 If~$\ep > 0$ is small enough, not only the set of open sites percolates, but also the set of closed sites does not percolate and the probability of a path of closed sites with length at least~$n$ starting from~$(0, 0)$ decays exponentially with~$n$ (see~\cite[Section 8]{berg_grimmett_schinazi_1998} for a proof).
 Because the players cannot appear spontaneously, the presence of a player in a space-time block~$B_{m, n}$ implies the existence of a path of closed sites to~$(m, n)$, which shows extinction.
\end{proof}


\noindent {\bf Acknowledgments}.
 The authors would like to thank two anonymous referees for their careful reading of a preliminary version of this paper and for providing helpful comments, especially for the study of the bistable region of the mean-field model.

 JK thanks the School of Mathematical and Statistical Sciences at Arizona State University for its hospitality and acknowledges the financial support of the Leibniz Association within the \textit{Leibniz Junior Research Group on Probabilistic Methods for Dynamic Communication Networks} as part of the Leibniz Competition and the German Research Foundation under Germany’s Excellence Strategy \textit{MATH+: The Berlin Mathematics Research Center}.


\end{document}